\documentclass[11pt, a4paper]{article} %twoside
\title{Non Hilbertian Lorentzian Length Spaces}
\author{Jona Röhrig}
\usepackage{amsmath,amsfonts,amstext,amsxtra,amssymb,amsthm,mathtools} % Math. Symbole
\usepackage{tensor}
\usepackage{hyperref}
\usepackage{wrapfig}
\usepackage{graphics}
\usepackage{graphicx}
\usepackage{subfiles}
\usepackage{MnSymbol}
\usepackage{graphbox}
\usepackage{witharrows}
\usepackage{xcolor}
\usepackage{tikz}
\usetikzlibrary{tikzmark}
\usetikzlibrary{decorations.pathreplacing, calc}
\usepackage{amsmath}

\usepackage[utf8]{inputenc} % Für ä's und so
\usepackage[T1]{fontenc}

\usepackage[english]{babel} % zum Beispiel "Contents" ist Englisch
\usepackage[margin=2.5cm]{geometry}

\usepackage{tikz}
\usetikzlibrary{matrix,quotes,calc,patterns,math,shapes}
\usepackage{float} % z.B. figures positionieren
\usepackage{verbatim}
\usepackage{thmtools}
\usepackage{caption}
\usepackage{subcaption}
\usepackage{wrapfig}
\usepackage{graphicx} % Bilder einbauen
\usepackage{enumitem}
\usepackage{mathrsfs}
\usepackage{stmaryrd}
\usepackage{lipsum}
\usepackage{bbm}
\usepackage{xcolor}
\usepackage{media9}
\definecolor{green}{HTML}{007F00}
\definecolor{red}{HTML}{7f0f0f}
\usepackage[makeroom]{cancel}
 \usepackage[T1]{fontenc}
  \usepackage{tikz-cd}

\usepackage{ifthen}
\usepackage{witharrows}
\usepackage{cite}

\usepackage{makeidx}
\makeindex
\usepackage[export]{adjustbox}
\hypersetup{
    colorlinks=true, 
    breaklinks=true, 
    citecolor=red, 
    linkcolor=black, 
    menucolor=red, 
    urlcolor=blue,
    bookmarksopen=false, 
    bookmarksopenlevel=0,
    pdftitle={Complete Minimal Surfaces and How to Print Them},
    pdfsubject={Complete Minimal Surfaces and How to Print Them},
    plainpages=false,% zur korrekten Erstellung der Bookmarks 
    hypertexnames=false,% zur korrekten Erstellung der Bookmarks 
}
\newcommand{\N}{\mathbb{N}} %natural numbers
\newcommand{\Z}{\mathbb{Z}} %integers
\newcommand{\R}{\mathbb{R}} %reals
\newcommand{\dd}{\text{d}} %differential

\newtheorem{definition}{Definition}[section]

\newtheorem{lemma}[definition]{Lemma}

\theoremstyle{plain} % definition
\newtheorem{remark}{Remark}

\newcommand\defindex[2][]{\textbf{#2}\ifthenelse{\equal{#1}{}}{\index{#2}}{\index{#1}}} %index entry+definition (bold text) where the index entry might differ from the inline text

\newcommand{\eqval}{\; \Leftrightarrow \;}
\newcommand{\imply}{\; \Rightarrow \;}

\newcommand{\mDef}{\coloneqq }
\newcommand{\pp}{\; \vcentcolon \;}

\newcommand\identity{1\kern-0.25em\text{l}}

\begin{document}
\maketitle
\begin{abstract}
In this note, the idea of finite-dimensional $L^p$ spaces is extended to the setting of Lorentzian length spaces to provide an example of a constant Finsler manifold that is nowhere Minkowskian. Looking at the sectional curvature bounds of this example leads to the more general statement that non-Hilbertian, normed spaces have no sectional curvature bounds. This statement holds in both the Riemannian and Lorentzian cases.
In addition, the Lorentzian $L^p$ spaces serve as an example in the context of Lorentzian Gromov-Hausdorff convergence. They show that unbounded sectional curvature or geodesic regularity is generally not preserved in the GH limit. 
This example also shows that Ricci bounds in the sense of TCD as well as bounds on the volume, the timelike diameter, and the dimension are not enough to ensure Gromov pre-compactness in the category of Lorentzian length spaces.
\end{abstract}

\section{Introduction}
Since the concept of Lorentzian length spaces was introduced in 2018 in \cite{kunzinger2018lorentz}, there have been multiple approaches to define the analogue of metric spaces in the setting of Lorentzian geometry. All capture the same idea of being the synthetic version of a Lorentzian manifold in the formulation of metric spaces. The distance between two causally related points (events) is given by the maximal time that can pass for an observer who visits both events. For spacetimes, this distance, called the time separation $\tau$, is given by the supremum of lengths over all causal curves connecting two events. \\
In the following, we will briefly revisit the definitions of Lorentzian (pre-) spaces \cite{kunzinger2018lorentz}, almost Lorentzian pre length spaces \cite{muller2019lorentzianGromovHausdorfAndFiniteness}, bounded Lorentzian spaces \cite{minguzzi2022lorentzian}, and Lorentzian spaces \cite{muller2024maximality} with a few remarks on their differences.
As a learning example, we will introduce the definition $(\R^{1,1},\tau^p)$ as an analogue to finite dimensional $L^p$ spaces, where the time separation function is given by $\tau(x,y)^p \mDef \left(\sup\{|x_0-y_0|^p - |x_1-y_1|^p,0\} \right)^{1/p}$ We show that although this is not a manifold, it still satisfies all the requirements of the definition mentioned above.
Calculating the sectional curvature of $(\R^{1,1},\tau^p)$, we find that although it is just two-dimensional, it has unbounded sectional curvature in both directions at every point.
This observation can be generalized in the following lemmas for the Riemannian and Lorentzian case, respectively:
\begin{lemma}
\label{lemParRm}
A normed vector space $(X,|\cdot|)$, where $|\cdot|$ does not come from an inner product, has neither upper nor lower sectional curvature bounds as a length space.
\end{lemma}
\begin{lemma}
\label{Lem_LorUnb}
Let $X$ be a vector space with a time separation function $\tau$ which is translation invariant and such that $\tau(0,x)$ is homogeneous in $x$.\\
If there exists a set of vectors $x,y\in X$ such that $x,y,x+y,x-y\gg 0$ and \\
$2 \tau(0,x)^2 + 2 \tau(0,y)^2 \neq \tau(0,x+y)^2 + \tau(0,x-y)^2$, then $(X,\tau)$ has neither upper nor lower timelike sectional curvature bounds.
\end{lemma}
These results have been formulated in a similarly way in the context of Finsler geometry by Braun and Ohta in \cite{braun2023optimaltransporttimelikelower}.\\
In addition to sectional curvature bounds, also the Gromov-Hausdorff distance was generalised into the Lorentzian setting in \cite{muller2019lorentzianGromovHausdorfAndFiniteness} and \cite{minguzzi2022lorentzian}, measuring how similar two almost Lorentzian pre length spaces are. This definition can be applied to a compactified version of $(\R^{1,1},\tau^p)$, namely the cylinder $Cyl^p\mDef ([0,1]\times S^1,\tau^p)$. As we show in lemma \ref{lem_GHPropOfCyl}, these spaces depend continuously on $p$ and diverge for $p\to\infty$.\\
Using a result from \cite{beran2024nonlineardalembertcomparisontheorem}, we can show that $Cyl^p$ satisfies $TCD(0,2)$. The Ricci curvature can therefore be considered as bounded from below by $0$, the dimension bounded from above by $2$. Since the volume and timelike diameter are also uniformly bounded, $Cyl^p$ shows that these conditions are not enough to ensure Gromov-Hausdorff pre-compactness.\\
Since the $TCD$ condition requires a background measure on the set, we finally raise the question of whether this can be generated using the Lorentzian Hausdorff measure and find that for $p\in (1,2)$, these spaces have a non-integer Hausdorff dimension.

\section{Preliminaries}
Since the time separation function $\tau$ in a spacetime does not satisfy the triangle inequality but instead satisfies a conditional reverse triangle inequality, it cannot be considered a metric. This causes difficulties in the generalization to length spaces since $\tau$ does not give rise to a metric topology.  For example, in a Lorentzian context, "$\epsilon$-balls" do not form small neighbourhoods around a point but rather large, hyperbolic-like shapes with the base point at their tip. Therefore, we cannot define the topology in the usual way based on $\tau$ and have to find some other solution.
The most common formulation, used in \cite{kunzinger2018lorentz}, circumvents this problem by providing not only a time separation function $\tau$ for the space but also an underlying metric $\dd$ which can be used to define the topology:
\begin{definition}
A \textbf{Lorentzian pre-length space} $(X,\dd,\ll, \leq, \tau)$ consists of a set $X$, metric $\dd$ (positive definite, symmetric, and satisfying the triangle inequality), two relations $\ll, \leq\ \subset X\times X$ and a function $\tau: X\times X \to [ 0,\infty]$, such that:\\
$
\begin{array}{ll}
\begin{array}{l}
\bullet\ \leq \text{ is reflexive and transitive (a pre-order)} \\
\bullet\ \ll \text{ is transitive and contained in }\leq
\end{array}
& \bigg\}\ (X,\ll,\leq) \text{ is a causal space}\\
\begin{array}{l}
\bullet\ \ \tau \text{ is lower semi-continuous w.r.t d}\\
\bullet\ \ \tau(x,z) \geq \tau(x,y) + \tau(y,z) \ \ \ \forall\ x\leq y \leq z \in X\\
\bullet\ \ x\nleq y \imply \tau(x,y)=0 \text{ and } \tau(x,y)>0 \eqval x\ll y
\end{array}
& \Bigg \}\ \tau \text{ is a time separation function}\\
\end{array}
$
\end{definition}
This is a generalization of the concept of causal spacetimes, where $\ll, \leq$ and $\tau$ are given by the Lorentz metric $g$. The metric $\dd$ is unnatural. Although it is common practice to equip spacetimes with an auxiliary Riemannian metric (for example in the limit curve theorem, see \cite{GLG} Lemma 14.2), there is no canonical way as there cannot exist an injective, topology-preserving functor from spacetimes to metric spaces (see \cite{muller2019lorentzianGromovHausdorfAndFiniteness}).\\
One possible simplification is defined by Minguzzi and Suhr in \cite{minguzzi2022lorentzian}.
They avoid the necessity of a metric for their definition of bounded Lorentzian metric spaces.
To impose some kind of regularity on the time function $\tau$, they require $\tau$ to ensure the existence of a 'suitable' topology:
\begin{definition}
A \textbf{bounded Lorentzian-metric space} $(X,\tau)$ is a set $X$ and a function $\tau: X\times X \to [0,\infty)$ satisfying the following properties:\\
$ 
\begin{array}{l}
\bullet\ \ \tau(x,z) \geq \tau(x,y) + \tau(y,z) \ \ \ \forall\ x,y,z \in X \pp \tau(x,y)>0 \text{ and } \tau(y,z)>0 \\
\bullet\ \ \tau\text{ distinguishes points which means that for every } x\neq y \in X \ \text{ there exists } z \in X \text{ such that } \\ \hspace{4.5mm} \tau(x,z)\neq\tau(y,z)\text{ or } \tau(z,x) \neq \tau(z,y)\\
\bullet\ \ \text{There exists a topology $T$ on $X$ such that $\tau:X\times X\to \R$ is continuous and for every }\epsilon>0 \pp \\ \hspace{4.5mm} \{(x,y)\in X\times X\pp \tau(x,y) \geq\epsilon\} \text{ is compact}
\end{array}
$
\end{definition}
A limitation of this definition is that it requires the space to be bounded. Since spacetimes are usually not compact—otherwise, they would have singularities or be non-causal—this definition covers only bounded subsets of spacetimes.
In a second paper \cite{bykov2024lorentzianmetricspacesghconvergence}, the authors build upon this definition and defined unbounded Lorentzian metric spaces to solve these issues by requiring that the space is locally (in every causal diamond) a bounded Lorentzian metric space.\\
Without requiring that the set $\{(x,y)\in X\times X\pp \tau(x,y) \geq\epsilon\}$ be compact one could always choose the discrete topology, making $\tau$ naturally continuous.
Another possible definition is given by the Lorentzian space in \cite{muller2024maximality}:
\begin{definition}
\label{defLorSp}
    A \textbf{Lorentzian space} $(X,\dd,\ll,\leq,\tau,T)$ is a set $X$, endowed with two relation $\ll$ and $\leq$, a time separation function $\tau:X\times X\to \R$, a topology $T$ and a function\\ $\dd:C(X)\to [0,\infty]^{X\times X}$ (where $C(X)$ is the set of all $T$-closed subsets of $X$) such that:\\
    $\begin{array}{l}
         \bullet\ \ (X,\ll,\leq) \text{ is a causal space}\\
         \bullet\ \ \tau \text{ is $T$ - lower semicontinuous}\\
         \bullet\ \ \forall x,y,x\pp x\leq y\leq z \implies \tau(x,z) \geq \tau(x,y) + \tau(y,z)\\
         \bullet\ \ x\nleq y \implies \tau(x,y)=0 \text{ and } \tau(x,y)>0 \iff x\ll y\\
         \bullet\ \ \forall \text{temporally compact set }U\in C(X) \implies \left.\dd_I\right|_{int(U)\times int(U)} \text{ is a finite metric and generates }T\\
         \bullet\ \ \forall x,y\in X \pp x\leq y (x\ll y) \exists \text{ a causal (timelike) curve } \gamma:x \rightlsquigarrow y\\
         \bullet\ \ \forall x\in X\exists \text{a nbh. }x\in U\in C(X) \text{ and } C>0 \pp \text{every causal curve in $U$ has }\dd_U \text{-length }<C\\
         \bullet \ \ \forall x\in X \exists \text{a nbh. } x\in U\in C(X) \pp \leq \cap (U\times U) \text{is closed in } U\times U
    \end{array}$\\
    A curve $\gamma:[a,b]\to X$ is called causal (timelike) in this definition if there exists a neighborhood $U\in C(X)$ of $\gamma([a,b])$ such that $\gamma$ is $\dd_U$ Lipschitz and for all $a\leq t_1<t_2\leq b$ we have $\gamma(t_1)\leq\gamma(t_2)$.
\end{definition}
The common denominator for the definitions above is a space which just requires the inverse triangle inequality. It was defined by Müller in \cite{muller2022gromovMetricAndDimension} as the \textbf{almost Lorentzian pre-length space}.
\begin{definition}
An \textbf{almost Lorentzian pre-length space} is a duple $(X,\tau)$ where $X$ is a set and $\tau: X\times X\to \R$ is a function which is:\\
$\begin{array}{l}
\bullet\ \ \text{antisymmetric}\\
\bullet\ \ \text{satisfies }\tau(x,z) \geq \tau(x,y) + \tau(y,z) \ \ \ \forall\ x,y,z \in X\pp \tau(x,y)>0 \text{ and }\tau(y,z)>0\\
\end{array}$
\end{definition}
The last definition we want to introduce is the \textbf{bare Lorentzian space} from \cite{muller2024maximality}. This combines the idea of an almost Lorentzian pre-length space and a Lorentzian space in the sense that it needs only the data of an almost Lorentzian pre-length space $(X,\tau)$, but require that it can be mapped by a certain functor (see Theorem 1 in \cite{muller2024maximality}) to a Lorentzian space. The idea of this functor is to provide the required relations, topology and metric (on every closed set $U\in C(X)$) completely by $\tau$. The local metric, which also defines the Lipschitz class on $U\in C(X)$, is for example given by the Noldus$^2$ metric, defined by $\dd^{N^2} (x,y) \mDef \sup_{z\in U}\{\tau(x,z)^2 - \tau(y,z)^2|\}$ and the relations by $x\ll y\iff \tau(x,y)>0$ and $x\leq y\iff I^+(y)\subset I^+(x)$ and $I^-(x)\subset I^-(y)$.
\\
Note that the convention for the sign of $\tau$ is not consistent throughout the different definitions. Some require it to be antisymmetric, others to be non negative. By taking the antisymmetrical version $\tau^a(x,y) = \tau(x,y) - \tau(y,x)$ of the non negative version, or the non negative version $\tau^{\geq 0}(x,y) = \sup\{\tau(x,y),0\}$ of the antisymmetric, it is clear that this difference is merely a notational inconvenience but makes no difference.\\
We can also require more regularity than in a Lorentzian pre-length space, for example, in the form of a \textbf{Lorentzian length space} (see Definition 3.22 of \cite{kunzinger2018lorentz}). One of the differences is that we require the time separation function to be intrinsic, meaning that the supremum of the length of causal curves between two points equals their time separation. \\
This is analogous to a length space, which is a metric space where the metric is intrinsic (meaning the distance of two points equals the infimum of length of curves, connecting these points.)\\
The interesting aspect of these spaces is that they allow for the notion of sectional curvature bounds. Before we introduce them, we will briefly introduce the Gromov-Hausdorff distance of almost Lorentzian pre-length spaces (and thereby also all the others).

\subsection{Gromov-Hausdorff distance}
The Gromov-Hausdorff (GH) distance is a notion known for some time in the setting of metric spaces and measures how far two metric spaces are from being isometric.
Recently, Müller \cite{muller2022gromovMetricAndDimension} and Minguzzi \& Suhr \cite{minguzzi2022lorentzian} independently extended the concept of Gromov-Hausdorff distance to the Lorentzian setting:
\begin{definition}
The \textbf{Gromov-Hausdorff distance} of two almost Lorentzian pre-length spaces $(X,\tau_X)$,$(Y,\tau_Y)$ is:
\begin{equation*}
d_{GH}^-((X,\tau_X),(Y,\tau_Y)) \coloneqq \frac{1}{2} inf\{dist(R)\ : \ R \in \text{Corr}(X,Y)\}
\end{equation*}
The infimum is taken over all correspondences $R$, which are all relations $\text{R}\subset X\times Y$ such that $\text{pr}_1(\text{R})=X$ and $\text{pr}_2(\text{R})=Y$, where $\text{pr}_i$ is the projection to the $i$'th component.\\
The distortion of the correspondence is defined by:
\begin{equation*}
\text{dist}(R) \coloneqq sup\{|\tau_X(x,x') - \tau_Y(y,y')|\ : \ (x,y),(x',y') \in \R\}
\end{equation*}
\end{definition}
The GH distance of two isometric spaces is zero and the triangle inequality holds naturally. Considering isometry equivalence classes of bounded Lorentzian metric space, positivity has been established in \cite{minguzzi2022lorentzian}, making $\dd_{GH}^-$ a metric.
\subsection{Sectional Curvature}
The notion of (timelike) sectional curvature bounds of Riemannian (Lorentzian) manifolds can be extended to the setting of (Lorentzian) length spaces using triangle comparison theorems \cite{Burago2001} (\cite{beran2023curvature}). 
To define sectional curvature bounds we use comparison spaces (two-dimensional spaces of constant sectional curvature $K\in \R$). In the Riemannian case, we denote them by $(\mathcal{S}^+_k,\Bar{\dd})$ (which are scaled versions of the sphere for $K>0$, the Euclidean plane for $K=0$ or the hyperbolic space for $K<0$) and in the Lorentzian case as $(\mathcal{S}^-_k,\Bar{\tau})$ (scaled version of the de-Sitter space, Minkowski space or anti de-Sitter space).
If we speak of comparison elements in what follows, we mean that we reconstruct some geometric object from a (Lorentzian) length space in a comparison space based on the distances. For example, a comparison triangle of a triangle $abc \subset X$ is a triangle in the comparison space $\Bar{a}\Bar{b}\Bar{c}\subset \mathcal{S}^\pm_k$ such that all side lengths are alike. A comparison point to a point on one of the edges (a distance realizers between two vertices) is defined by having equal distance to the two neighboring vertices. (Here, we have assumed that the metric or time separation is strictly intrinsic, meaning that a distance realizer between points exist. Without this assumption, the definition becomes a bit more technical.)
\begin{definition}
A length space $(X,\dd)$ has a sectional curvature bound below (above) by $K\in\R$, iff for every point $p\in X$ there exists a neighborhood $p \in U \subset X$, such that for any triangle $abc \subset U$ and any point $M \in bc$, for its corresponding comparison triangle $\Bar{a}\Bar{b}\Bar{c}\subset \mathcal{S}^+_k$ and point $\Bar{M}\in \Bar{b}\Bar{c}$ we find that $\dd(a,M) \geq \Bar{\dd}(\Bar{a}\Bar{M})$ $(\leq)$.
\end{definition}
The idea in the Lorentzian setting is the same except that we have to demand the triangle to be timelike ($a\ll b\ll c$) in order to be able to measure the time separation of the edges. The time separation from one vertex to the opposite edge is then compared with the distance in a comparison triangle in $\mathcal{S}^-_k$. The formal definition, which includes some additional technical requirements, can be found in Definition 3.2 of \cite{beran2023curvature}.

\subsection{Ricci Curvature, Dimension and Volume}
In the setting of metric measure spaces, there exists synthetic versions of Ricci curvature bounds, called curvature dimension condition $CD(K,N)$. This condition determines the curvature by looking at convexity properties of the entropy of the optimal transports of probability measures (\cite{Villani2009}).
Recently, $CD$-conditions have been generalized to the Lorentzian framework in form of the timelike curvature dimension condition $TCD(K,N)$ \cite{Mondino_2022}, \cite{cavalletti2023optimaltransportlorentziansynthetic}. We will not explain the details here but refer to the sources named above for insights.\\
To calculate $CD$ or $TCD$, we need a background measure on a (Lorentzian) metric spaces. One way to define a measure on bare Lorentzian pre-length spaces is the Lorentzian Hausdorff measure. This has been defined by McCann and Sämann in \cite{McCann_2022} and in an almost identical way by Müller in \cite{M_ller_2023}:
\begin{definition}
    For $N\in[0,\infty)$, the $N$-dimensional Hausdorff volume of a causal diamond $J(p,q)$ for $p\ll q$ is defined by 
    \begin{equation*}
        \rho_N (J(p,q)) \mDef \omega_N \tau(p,q)^N \qquad \qquad \omega_N \mDef\frac{\pi^{\frac{N-1}{2}}}{N \Gamma(\frac{N+1}{2}) 2^{N-1}}
    \end{equation*}
    where $\Gamma(N)$ is Euler's gamma function.\\
    For a Borel measurable set $A\subset X$ and $\delta>0$, we define
    \begin{equation*}
        \nu_\delta^N(A)\mDef \inf\{\sum_{i=0}^\infty \rho_N(J(p_i,q_i)) \pp p_i\leq q_i\in X, \text{diam}_\dd(J(p_i,q_i))<\delta \text{ and } A\subset \bigcup_{i=1}^\infty J(p_i,q_i)\}
    \end{equation*}
    The Lorentzian $N$-dimensional Hausdorff measure is defined by looking at the limit $\delta \to 0$ (which is well defined since a smaller delta restricts the allowed coverings from the definition above, making $\nu_\delta^N(A)$ monotonically increasing in $\delta$).
    \begin{equation*}
        \nu^N(A) \mDef \sup_{\delta>0} \nu_\delta^N(A)
    \end{equation*}
    The geometric Lorentzian Hausdorff dimension of $A$ is 
    \begin{equation*}
        \dim^\tau(A)\mDef \inf\{N\geq 0\pp \nu^N(A)<\infty\}.
    \end{equation*}
\end{definition}

\section{Lorentzian Analogue of Finite Dimensional \texorpdfstring{$L^p$}{TEXT} Spaces}
In the following, we transfer the idea of finite dimensional $L^p$ spaces to the Lorentzian setting to obtain an example that yields many insights into properties of the definitions given above.
\begin{definition}
For $1\leq p\leq \infty$ and $x,y\in\R^{1,n}$, we define the non-negative, homogeneous function
$$|x|^p \mDef
\begin{cases}
\left(|x_0|^p - \sum_{i=1}^n |x_i|^p \right)^{1/p} & |x_0|^p \geq\sum_{i=1}^n |x_i|^p \text{ and } x_0>0\\
0 & \text{ else } \\
\end{cases}$$
The corresponding time distance function is $\tau^p(x,y)\mDef |y-x|^p$.\\
The relations $\ll$ and $\leq$ are inherited from Minkowski space.\\
For $\tau^\infty$ we take the point-wise limit $p \to \infty$ of $\tau^p$.
\end{definition}
In the following, we will restrict our discussion to the two dimensional case. 
First, we will check in which categories of 'Lorentzian spaces' $(\R^{1,1},\tau^p)$ is.
We start by proving the foundation, namely the inverse triangle inequality:
\begin{lemma}
For all $x,y\in\R^{1,1}$ with $x, y \gg 0$, the reverse triangle inequality \\$|x+y|^p \geq |x|^p+|y|^p$ holds for all $p\in [1,\infty]$.\\
Thereby, $(\R^{1,1},\tau^p)$ is an almost Lorentzian pre-length space.
\end{lemma}
\begin{proof}
Using the homogeneity of $|\cdot|^p$, we can scale $x,y$ such that we can assume $|x|^p=\lambda$ and $|y|^p=1-\lambda$. Because we assumed $0\ll x,y$, we can assume that we can choose $\lambda\in(0,1)$.
The inequality to prove thereby reads $ |x+y|^p \geq \lambda+ 1-\lambda =1$.
By scaling $x,y$ to unit length, $\hat{x}\mDef\frac{x}{\lambda}$ and $\hat{y}\mDef\frac{y}{1-\lambda}$ we have to show the concavity of $|\cdot|^p$ on the unit hyperboloid: 
\begin{equation*}
    |\lambda\hat{x} + (1-\lambda)\hat{y})|^p \geq 1 = \lambda |\hat{x}|^p + (1-\lambda) |\hat{y}|^p \text{ on }\{ x\in\R^{1,1} \pp |x|^p=1 \}
\end{equation*}
Solving $\left(|x_0|^p-|x_1|^p\right)^{1/p}\overset{!}{=}1$ for $x_0$ gives $x_0 = \left( 1+|x_1|^p \right)^{1/p}$.
For any $x_1\neq 0$, we can take the second derivative w.r.t. $x_1$ and obtain
$\frac{\partial^2 x_0}{\partial x_1^2} = (p-1) \left| x_0\right| {}^{p-2} \left(\left| x_0\right| {}^p+1\right){}^{\frac{1}{p}-2} > 0$. Together with the continuity at $x_1=0$,
the fact that $\lim_{x_1 \nearrow 0} \frac{\partial x_0}{\partial x_1} \leq \lim_{x_1 \searrow 0} \frac{\partial x_0}{\partial x_1}$ and that $|x|^p$ is monotonically increasing in $x_0$ proves the convexity of $|\cdot|^p$.
    Hence, $|\cdot|^p$ satisfies the inverse triangle inequality and thereby also $\tau^p$, which makes $(\R^{1,1},\tau^{p})$ an almost Lorentzian pre-length space. 
\end{proof}
Next, we prove the much stronger requirements for a Lorentzian (pre-)length space:
\begin{lemma}
\label{lem_isLorLenSP}
$(\R^{1,1},\dd,\ll,\leq,\tau^p)$, where $\dd$ is the the Euclidean metric on $\R^2$, is a Lorentzian length space for any $p\in [0,\infty)$ and a Lorentzian pre-length space for $p=\infty$.
\end{lemma}
\begin{proof}
$\tau^p =\begin{cases}
\left( |y_0 - x_0|^p - |y_1 - x_1| ^p \right)^{1/p} & (y_0 - x_0) \geq |y_1 - x_1| \\
    0 & \text{ else }
        \end{cases} $
    is continuous for any $p\in[0,\infty)$ as a function $\tau\pp \R^4\to\R$. For $p=\infty$, the pointwise limit of the distance is
    \begin{equation*}
        \tau^\infty (x,y) = \begin{cases}
        y_0 - x_0 & y_0 - x_0 > |y_1 - x_1|\\
        0 & y_0 - x_0 \leq |y_1 - x_1|
        \end{cases}
        \end{equation*}
    $\tau^\infty$ is not continuous, but lower semi-continuous because the upper case in the definition of $\tau^\infty$ is defined by an open relation.\\
    Together with the previous lemma and the fact that $(\leq,\ll)$ are compatible with $\tau^p$, we have shown that $(\R^{1,1},\tau^p)$ is a Lorentzian pre-length space. To show that this space is also a Lorentzian length space (see Definition 3.22 in \cite{kunzinger2018lorentz}) we first note that the reverse triangle inequality can be used to argue that lines are distance realizers, meaning the longest causal curves between any two causal points. Moreover, the length of a line equals the distance of its endpoints (this follows from the homogeneity of $|\cdot|^p$). Hence, $\tau^p$ is intrinsic.
The remaining requirements for a Lorentzian length-space, which are\\
    \null\hspace{1cm} 1) locally causally closed\\
    \null\hspace{1cm} 2) causally path connected\\
    \null\hspace{1cm} 3) localizable\\
    are all satisfied for all $p\in [1,\infty)$ (by arguing with the given structure of Minkowski space and the continuity of $\tau^p$). For $p=\infty$ only the localizability is not satisfied. This can be seen using proposition 3.17 from \cite{kunzinger2018lorentz}, since $L_{\tau^\infty}$ is not upper semi-continuous (consider for example $\gamma_s(t)\mDef (t, ts)$ for $s\to 1$.)
\end{proof}
\begin{lemma}
    $(\R^{1,1},\dd,\ll,\leq,\tau^p,T)$ is a Lorentzian space for $p\in[1,\infty]$ with $T$ being the metric topology from $\R^2$ and $\dd$ being the Euclidean metric on $\R^2$ for all $U\in C(X)$.
\end{lemma}
\begin{proof}
     All the requirements for a Lorentzian space are satisfied naturally (since everything not involving $\tau^p$ is equivalent to the Minkowski space) or have already been shown.
\end{proof}
The space $(\R^{1,1},\tau^p)$ can also be used to show that not every Lorentzian space is also a bare Lorentzian space, although the latter requires much less data. The problem is that, given the functor $C$ from theorem 1 of \cite{muller2024maximality}, the local metric we get is not the same as in the lemma above and does not lie in the same Lipschitz class:
\begin{lemma}
    The space $(\R^{1,1},\tau^p)$ is not a bare Lorentzian space for $p>2$.
\end{lemma}
\begin{proof}
    First, we calculate the Noldus$^2$ metric for $\tau$ inside a causal diamond. We can showcase what happens using the diamond $U=J(p,q)$ for $p=(-1,0)$ and $q=(1,0)$ which is a closed set in the topology defined in \cite{muller2024maximality}.
    The Noldus$^2$ metric is defined by\\
    $\dd^{N^2}(x,y) = \sup_{z\in U} \left\{ |\tau^p(x,z)^2 - \tau^p(y,z)^2| \right\}$ (where $\tau^p$ is the antisymmetric time separation function).
    To show that $\dd^{N^2}$ is not strongly equivalent to the Euclidean metric $\dd$ from $\R^2$, we estimate the distance of $x=(0,0)$ and $y=(0,\lambda)$ from below by picking $z = (\frac{1}{2},\frac{1}{2})\in U$. Since $\tau^p(x,z)=0$, we can write
    \begin{equation*}
            \dd^{N^2} (x,y) \geq \tau^p(y,z)^2 = \frac{1}{4} \left( \frac{1}{2^p} - \left(\frac{1}{2} - \lambda\right)^p\right)^{2/p} = \frac{1}{16}\left( 1 - (1 - 2\lambda)^{p} \right)^{2/p}
    \end{equation*}
    This distance depends on $\lambda$ (and thereby on $\dd$) with infinite steepness for $\lambda\to 0$:
    $$\frac{\partial}{\partial \lambda} \frac{1}{16}\left( 1 - (1 - 2\lambda)^{p} \right)^{2/p} = \frac{1}{4} \left(\underbrace{1-(1-2 \lambda)^p}_{\xrightarrow{\lambda \to 0} 0}\right)^{2/p-1}\underbrace{(1-2 \lambda)^{p-1}}_{\xrightarrow{\lambda \to 0} 1}  \xrightarrow{\lambda \to 0, \ p>2} \infty $$
    The issue here is that every non-constant curve is not Lipschitz continuous, which implies that there are no causal curves (in the sense defined in Definition \ref{defLorSp}). Hence, the space would be completely causally path unconnected. 
\end{proof}
Since $(\R^{1,1},\tau^p)$ is not bounded (compact), it is not a bounded Lorentzian metric space. Therefore, we restrict the space to some interval in time direction and, to eliminate timelike boundaries, glue it to a cylinder $Cyl^p \mDef([0,1]\times S^1,\tau^p)$. For this cylinder we find:
\begin{lemma}
$Cyl^p$ is a bounded Lorentzian metric space for $p\in[1,\infty)$ but not for $p=\infty$.
\end{lemma}
\begin{proof}
First we observe that $\tau^p$ distinguishes points for any $p\in[1,\infty]$. This follows from the observation, that $\tau^p$ is compatible with the relation $\ll$ which is inherited from the Lorentzian cylinder. Since the Lorentzian cylinder is distinguishing, so is $\ll$ in $Cyl^p$ and therefore also $\tau^p$.\\
For $p<\infty$, a topology which fits the requirements is the manifolds topology inherited from the Lorentzian cylinder. This follows from the observation in the proof \ref{lem_isLorLenSP} that $\tau^p$ is continuous in this topology for $<\infty$ and that the sets $\{(x,y)\in Cyl^p\times Cyl^p \pp \tau^p (x,y)\geq \epsilon\}$ are closed sets in a space $Cyl^p\times Cyl^p$ which is compact.\\
For $p=\infty$, we have to show that there can not exist a suitable topology in order to show that $Cyl^\infty$ is not a bounded Lorentzian metric space. \\
For this, consider the set $\{(1,\phi)\}\in Cyl^\infty$ for some $\phi\in S^1$ and assume there would exist a suitable topology $T$.
If $\{(1,\phi)\}$ is open, sets of the form $\{(x,y)\pp \tau^\infty\geq \epsilon\}$ are not compact anymore, since any sequence converging to $(1,\phi)$ in the metric topology from before, does not converge in $T$ (except if they are constant at some point). Therefore we find sequences without converging subsequence. Hence $\{(1,\phi)\}$ can not be open. On the other hand, using the notion $\tau_x(y)\mDef\tau(x,y)$, we can express $\{(1,\phi)\}$ as $\left(\tau^\infty_{(0,\phi+1)}\right)^{-1} \left([0,\frac{1}{2})\right)\cap \left(\tau^\infty_{(0,\phi-1)}\right)^{-1} \left([0,\frac{1}{2})\right)$ which means that, by the continuity of $\tau^\infty$, we can deduce that $\{(1,\phi)\}$ must be open since $[0,\frac{1}{2})$ is open in $[0,\infty)$.
\end{proof}
\begin{remark}
For $p=1$ and $p=\infty$, the inverse triangle inequality becomes an equality for sets of non-collinear vectors. Consider for example $|(3,1)|^l_1 = 2 = |(1,0)|^l_1 + |(2,1)|^l_1$. For $p=\infty$, for timelike vectors, $|\cdot|^p$ depends only on the time component and the triangle inequality is an equality for any set of timelike vectors. This means that geodesics split.
\end{remark}
Looking at the timelike sectional curvatures of $(\R^{1,1},\tau^p)$, we find:
\begin{lemma}
For any $p \in [1,\infty),\ p\neq 2$, the space $(\R^{1,1},\tau^p)$ has neither an upper nor a lower sectional curvature bound. For $p=2$ and $p=\infty$, the space is flat.
\end{lemma}
\begin{proof}
For $p=2$, $(\R^{1,1},\tau^p)$ is the Minkowski space and therefore flat.
For $p=\infty$, any timelike triangle $abc$ satisfies $\tau^\infty(a,b) + \tau^\infty(b,c) = \tau^\infty(a,c)$. Thus, the comparison triangle in $\R^{1,1}$ is degenerate and $\tau^\infty(aM) = |\Bar{a}\Bar{M}|$.\\
   For any other $p$, we use lemma \ref{Lem_LorUnb}:\\
   Consider $x=(2,0)$ and $y = (1,\frac{1}{4})$.
   We observe that $x,y,x+y,x-y\gg0$ and calculate (using the intermediate value theorem) 
   \begin{align*}
      E(p) &= 2 (|x|^p)^2 + 2 (|y|^p)^2 -  (|x+y|^p)^2 + |x-y|^p)^2\\
      &= 8 + 2 \left(1 - \frac{1}{4^p}\right)^{2/p} -  (3^p - \frac{1}{4^p})^{2/p} -  (1 - \frac{1}{4^p})^{2/p}\\
       &= 8 + f_p\left(1 - \frac{1}{4^p}\right) \textcolor{gray}{-f_p(1) + f_p(1) -f_p(3) + f_p(3)} -  f_p\left(3^p - \frac{1}{4^p}\right)   &\null \qquad f_p(x)\mDef(x)^{2/p} \\
      &= 8 - \frac{f_p'(x_1)}{4^p} + 1 - 9 - \frac{f_p'(x_2)}{4^p}  &\null \qquad x_1\in (1-\frac{1}{4^p},1)\\
      &= \frac{f_p'(x_1) - f_p'(x_2)}{4^p} &\null \qquad x_2\in (3-\frac{1}{4^p},3)
   \end{align*}
   Since $f_p(\cdot)$ is convex for $p>2$ and concave for $p<2$ (which means that $f'_p(\cdot)$ is either monotonically increasing or decreasing), $E(p)\neq0$ for any $\in (1,\infty), p\neq 2$. This proves that the parallelogram law is not satisfied which is the conditions of \ref{Lem_LorUnb}.
\end{proof}
The statement above can be analogously expressed in the Riemann $(\R^2,l^p)$ case.\\
Having demonstrated a use case of lemma \ref{lemParRm} and \ref{Lem_LorUnb}, it is now time to prove these statements:
\begin{proof} \textit{of lemma} \ref{lemParRm}: 
First, note that $(X,\dd)$ is a length space with $\dd(x,y)\mDef|y-x|$:\\ It is clear that $\dd$ is a metric.
Due to the triangle inequality, straight lines are always (one of) the shortest paths between any two points. Since $|\cdot|$ is linear, the length of a straight line is equal to the distance between its endpoints. Therefore, $\dd$ is a strictly intrinsic metric.\\
Next, we argue that there cannot be any bound of the sectional curate of $(X,\dd)$.
Since $|\cdot|$ does not come from a scalar product, we know that the polarization formula $\langle x,y\rangle \mDef \frac{1}{2}\left(|x+y|^2 -|x|^2 - |y|^2\right)$ can not define a scalar product. Since positivity and symmetry are satisfied automatically by $\langle \cdot,\cdot\rangle$, the linearity must be violated. This leads to the conclusion, that $|\cdot|$ cannot satisfy the parallelogram law $2|x|^2 + 2 |y|^2 = |x+y|^2 + |x-y|^2$ for $x,y\in X$.\\
Let us assume (WLOG as we will see) that we have some $x,y\in X$ such that\\ $2|x|^2 + 2 |y|^2 < |x+y|^2 + |x-y|^2$. Representing $x$ and $y$ by $u, v \in X$ with $x=u+v$ and $y=u-v$, we obtain $|u+v|^2 + |u-v|^2 < 2 |u|^2 + 2|v|^2$. Therefore, the parallelogram law is always broken in both directions. \\
Because $\dd$ is translation invariant, we calculate the curvature at $p=0$ and show that the assumption of a curvature bound leads to a contradiction.
Let $U\subset X$ be some open neighborhood of $0$ in which we assume the triangle comparison curvature bound to hold.
By scaling the vectors $x$ and $y$ down, we can assume that the following argument works completely inside $U$.
\begin{center}
\begin{tikzpicture}
% Define the vertices of the parallelogram
\coordinate (A) at (0,0);
 \coordinate (B) at (4,0);
 \coordinate (C) at (5,2);
 \coordinate (D) at (1,2);

 % Draw the parallelogram
 \draw (A) -- (B) -- (C) -- (D) -- cycle;

 % Label the sides
 \node at ($(A)!0.5!(D)$) [left] {\( y \)};
 \node at ($(A)!0.5!(B)$) [above] {\( x \)};

 % Draw the diagonal x-y
 \draw[dashed] (B) -- (D);

 % Draw line from origin to midpoint of diagonal
 \coordinate (M) at ($(B)!0.5!(D)$);
 \draw[->] (A) -- (M);
 \node[rotate=-35] at (2.7,1.2) {\( \frac{1}{2}(x + y) = M\)};
 % Label the midpoint
 \node at (A) [left] {\( O = a \)};
 \node at (B) [right] {\( b \)};
 \node at (D) [left] {\( c \)};
 \end{tikzpicture}
\end{center}
The inequality $\frac{1}{2}|x|^2 + \frac{1}{2}|y|^2 -|\frac{1}{2}(x-y)|^2 < |\frac{1}{2}(x+y)|^2$ can be reformulated in terms of the triangle $\triangle abc$ (see graphic above). Using the abbreviation $pq \mDef \dd(p,q)$ for distances of points, we can write:
$aM = \epsilon + \sqrt{\frac{1}{2}ab^2 + \frac{1}{2} ac^2 - \frac{1}{4} bc^2}$ for some $\epsilon >0$.
We know by the linearity of $|\cdot|$ that scaling all vectors down by a factor $\lambda\in(0,1]$ (we will write $\dd(\lambda p, \lambda q) = pq_\lambda$), causes also $\epsilon$ to scale linearly with respect to $\lambda$:
$aM_\lambda = \lambda \epsilon + \sqrt{\frac{1}{2}ab_\lambda^2 + \frac{1}{2} ac_\lambda^2 - \frac{1}{4} bc_\lambda^2}$.\\
The idea of the proof is that, if we reconstruct a comparison triangle $\triangle \bar{a}\bar{b}\bar{c}$ of $\triangle abc$, the difference $\Bar{\epsilon} \mDef \bar{a}\bar{M}$ will scale not linear but quadratically in $\lambda$ since manifolds are locally approximately flat:\\
Consider a sufficiently small triangle $abc$ in a Riemannian manifold $(M,g)$. Choosing normal coordinates $x^i$ around the point $a$, we can expand the metric in these coordinates as \\
$g_{\mu \nu}(x) = \delta_{\mu \nu} + \R_{\mu \mu' \nu \nu'} x^{\mu'} x^{\nu'} + O(|x|^3)$, hence $g_{\mu \nu}(\lambda x) = \delta_{\mu \nu} + O(|\lambda|^2)$. If we fix only $a\in M$ and determine $b$ and $c$ by a given set of side length $\lambda \text{ab}$, $\lambda \text{ac}$, $\lambda\text{bc}$, we can map these three side lengths to the length $\dd(a,M)$, where $M$ is the midpoint of $\text{bc}$.
 Using the expansion of the metric, we can expect that this distance differs from the distance in the flat case in orders of $\lambda^2$.\\
 We now verify this idea by constructing a comparison triangle with side lengths \(ab = \bar{a}\bar{b}\), \(ac = \bar{a}\bar{c}\), and \(bc = \bar{b}\bar{c}\) in \(\mathcal{S}^-_k = \frac{1}{\sqrt{k}}S^2\) for \(k>0\) and calculate the distance \(\bar{a}\bar{M}\). The case \(k<0\) works analogously. (Actually, the calculation is exactly the same when considering that \(\cos(ix) = \cosh(x)\) and \(\sin(ix) = -i\sinh(x)\)).\\
 To construct a comparison triangle in $\frac{1}{\sqrt{k}}S^2\subset \R^3$, we start by setting $\Bar{a} \mDef (\frac{1}{\sqrt{k}},0,0)$ (by arbitrary choice), $\Bar{b} \mDef \frac{1}{\sqrt{k}}(\cos(ab \sqrt{k}),\sin(ab\sqrt{k}) ,0)$ (by demanding the distance $ab = \Bar{a}\Bar{b}$, given by $\Bar{\dd}(\Bar{p},\Bar{q}) = \frac{1}{\sqrt{k}}\sphericalangle \Bar{p}\Bar{q}$) and \\ $\Bar{c} \mDef \frac{1}{\sqrt{k}}\left\{\cos (ac \sqrt{k}), \frac{\cos (bc \sqrt{k})}{\sin (ab \sqrt{k})} - \frac{\cos (ac \sqrt{k})}{\tan (ab \sqrt{k})},\sqrt{1-(\frac{\cos
    (bc \sqrt{k})}{\sin (ab \sqrt{k}) } - \frac{\cos (ac \sqrt{k})}{\tan (ab \sqrt{k})})^2-\cos ^2(ac \sqrt{k})}\right\}$ as a solution to $\sphericalangle \Bar{a}\Bar{c} = ac$ and $\sphericalangle \Bar{b}\Bar{c} = bc$. (These equations are satisfied only for sufficiently small $ab$, $ac$, and $bc$. Because we are interested in the small limit later, this shall not bother us. Additionally, this definition is just a valid solution in the case of $1-(\frac{\cos
    (bc \sqrt{k})}{\sin (ab \sqrt{k}) } - \frac{\cos (ac \sqrt{k})}{\tan (ab \sqrt{k})})^2-\cos ^2(ac \sqrt{k})\geq 0$, otherwise we would have to introduce a sign. To not complicate the calculation by introducing case distinctions, we will always assume to be in this case. The other case works completely the same.)\\
    First, we observe that $|\Bar{c}|^2 = \frac{1}{k}$. Next, we calculate $\Bar{a}\Bar{c} = \frac{1}{\sqrt{k}}\sphericalangle \Bar{a}\Bar{c} = \frac{1}{\sqrt{k}}\arccos\left( \frac{k\cos (ac \sqrt{k})}{k} \right) = ac$.
    And finally
    \begin{align*}
     \Bar{b}\Bar{c} &=\frac{1}{\sqrt{k}}\arccos\left( \cos(ab\sqrt{k}) \cos(ac\sqrt{k}) + \sin(ab\sqrt{k}) \left(\frac{\cos (bc\sqrt{k})}{\sin (ab\sqrt{k})} - \frac{\cos (ac\sqrt{k})}{\tan (ab\sqrt{k})}\right)\right)\\
    &=\frac{1}{\sqrt{k}}\arccos\left( \cos(ab\sqrt{k}) \cos(ac\sqrt{k}) + \cos (bc\sqrt{k}) - \cos (ac\sqrt{k})\cos (ab\sqrt{k})\right)\\
     &= bc
     \end{align*}
    Let $\Bar{M}$ be the midpoint of $\Bar{b}\Bar{c}$ (we can calculate it in $\R^3$ and then scale it to the $\frac{1}{\sqrt{k}}S^2$). The comparison distance of $aM$ is
    \begin{align*}
     \scalebox{0.75}{%
         $%
        \begin{aligned}
             \Bar{a}\Bar{M} &= \frac{1}{\sqrt{k}}\cos^{-1}\left(\frac{\cos(ab \sqrt{k}) + \cos(ac\sqrt{k})}{\sqrt{\left(\cos (ac\sqrt{k}) + \cos (ab\sqrt{k})\right)^2 +\left(\frac{\cos (bc\sqrt{k})}{\sin (ab\sqrt{k})} - \frac{\cos (ac\sqrt{k})}{\tan (ab\sqrt{k})}+ \sin (ac\sqrt{k})\right)^2 + \sqrt{1-(\frac{\cos
            (bc)}{\sin (ab\sqrt{k}) } - \frac{\cos (ac\sqrt{k})}{\tan (ab\sqrt{k})})^2-\cos ^2(ac\sqrt{k})}^2}}\right)\\
            &= \frac{1}{\sqrt{k}}\cos^{-1}\left(\frac{\cos (ab\sqrt{k})+\cos (ac\sqrt{k})}{2 \cos \left(\frac{bc\sqrt{k}}{2}\right)} \right)
         \end{aligned}%
         $%
     }
\end{align*}
Scaling all distances $pq \to \lambda\cdot pq = pq_\lambda $ as above and expanding the distance $\Bar{a}\Bar{M}_\lambda$ in terms of $\lambda$ for $\lambda \searrow 0$ we get:
    $$\Bar{a}\Bar{M}_\lambda = 0 + \lambda \cdot \frac{\sqrt{2ab^2 + 2ac^2 - bc^2}}{2} -
   \frac{\lambda^3}{6} k \underbrace{\frac{(ab - ac - bc) (ab + ac - bc) (ab - ac + bc) (ab + ac + bc)}{8 \sqrt{2 (ab^2 + ac^2) - bc^2}}}_{>0 \text{   since $S^2$ has positive sectional curvature}}
   + O(\lambda^4)$$
Contrary to our expectations, the difference $\Bar{\epsilon}_\lambda$ here does not scale quadratically with $\epsilon$, but even cubical.\\
Concluding, we can write for $\lambda$ small enough:
$$aM_\lambda = \epsilon_\lambda + \sqrt{\frac{1}{2}ab_\lambda^2 + \frac{1}{2} ac_\lambda^2 - \frac{1}{4} bc_\lambda^2} = \Bar{a}\Bar{M}_\lambda + \epsilon_\lambda - \Bar{\epsilon}_\lambda> \Bar{a}\Bar{M}_\lambda$$
showing that $(X,\dd)$ has no sectional curvature bound below by $k$ for any $k>0$. Taking $k<0$ and switching $x,y$ for $u,v$ gives the same result in the other direction, proving the claim. 
\end{proof}

In the Lorentzian case, that is, if we are given a vector space with a signature $(1,n)$ product $(X,\langle\cdot,\cdot\rangle)$, the time separation function is given by 
$\tau(x,y)\mDef\begin{cases}
     \sqrt{- \langle y-x, y-x\rangle}& x\ll y\\
    0 & \text{else}
\end{cases}$.\\
Consequently, the polarization formula $\langle x,y\rangle = \frac{1}{2}\left( \tau(0,x)^2 + \tau(0,y)^2 - \tau(0,x+y)^2 \right)$ holds only for $0\ll x,y$ (and therefore also $0\ll x+y$).\\
Thereby, the parallelogram law $2 \tau(0,x)^2 + 2 \tau(0,y)^2 = \tau(0,x+y)^2 + \tau(0,x-y)^2$ holds true only if $x,y,x+y,x-y\gg 0$. This is the reason for the formulation of lemma \ref{Lem_LorUnb}. The rest of the proof of lemma \ref{Lem_LorUnb} is analogous to the proof of lemma \ref{lemParRm}.
\begin{remark}
Lemma \ref{lemParRm} and \ref{Lem_LorUnb} hold not only for vector spaces but also for spaces that locally look like a normed vector space, for example Finsler manifolds (see \cite{braun2023optimaltransporttimelikelower}).
\end{remark}
\subsection{Gromov-Hausdorff Continuity of \texorpdfstring{$(\R^{1,1},\tau^p)$}{TEXT}}
We can use $(\R^{1,1},\tau^p)$ also as an example for Gromov-Hausdorff continuity and convergence.\\
If we consider the entire space $\R^{1,1}$, there is little chance of proving GH continuity with respect to $p$. Therefore we take a look at the cylinder $Cyl^p$.
To motivate how the GH-distance behaves, we consider the identity relation $\text{id} = \{(x,x)\pp x\in Cyl^p\}\in Corr(Cyl^p,Cyl^q)$ and calculating the distortion for a pair of points on the past and future boundary $x^- \mDef (0,0)\in \partial^-Cyl^p$ and $x^+ \mDef (1,\phi)\in \partial^+ Cyl^p$ (for $\phi\geq 0$):

$$\left| \tau^p(x^-,x^+) - \tau^q(x^-,x^+)\right| = \left| (1-\phi^p)^{1/p} - (1-\phi^q)^{1/q} \right|$$

The function $ (1-\phi^p)^{1/p}$ depends continuously on $p$ and has a pointwise limit for $p\to\infty$, but does not converge uniformly. This behavior can also be found in general, looking at $Cyl^p$:

\begin{lemma}
    \label{lem_GHPropOfCyl}
    $Cyl^p$ is GH continuous in $p$ for all $p\in[1,\infty)$.
    For $p\to\infty$, the sequence $\{Cyl^p\}_p$ diverges and is therefore not precompact, also the pointwise limit of $\tau^p$ exists. 
\end{lemma}
\vspace{0 cm}
\begin{proof}
Because of the homogeneity of $\tau^p$, the difference between $\tau^p$ and $\tau^q$ is the greatest between $x\in \partial^- Cyl^p$ and $y\in \partial^+ Cyl^p$ (using the identity correspondence). We can abbreviate $y-x \mDef (1,\phi)$ for $\phi \in (-\pi,\pi)$. Since $\tau^p(x,y)$ is zero for $x,y$ spacelike separated for any $p$ and because of $\Z_2$ symmetry, we can further assume $\phi\in (0,1)$ and calculate the distortion of the identity as an upper bound for $2\cdot \dd_{GH}$:
\begin{equation*}
2 \cdot \dd_{GH}(Cyl^p,Cyl^q) \leq \text{dist}(\text{id}) = \sup_{x,y\in Cyl^l}\left|\tau^p(x,y) - \tau^q(x,y)\right| = \sup_{\phi\in (0,1)}\left| \left( 1 -|\phi|^p \right)^{1/p} - \left( 1 -|\phi|^q \right)^{1/q}  \right|
\end{equation*}
For fixed $\phi$ and $p$, this is Lipschitz continuous in $q$ and vanishes for $p=q$. Therefore, the maximum in $\phi$ is Lipschitz continuous in $p$ and $q$ and vanishes for $p=q$.
This means that in the GH metric, the spaces $Cyl^p$ are continuous w.r.t $p$.\\
To show that $Cyl^p$ for $p \to \infty$ diverges, we use a result from \cite{muller2022gromovMetricAndDimension}.\\
\begin{minipage}[t]{0.68\textwidth}
 In theorem 3 it is stated that the Lorentzian GH distance is bounded below by $\frac{1}{2}\Phi^*\dd_{GH}^+$. This is the GH distance when mapping an almost Lorentzian pre-length space to a metric space using the Noldus metric \\
 $\dd_{\tau^p}(x,y) \mDef \sup_{z\in Cyl^p} \left| \tau^p(x,z) - \tau^p(y,z) \right|$ for $x,y\in Cyl^p$.\\
First, we estimate the Noldus distance for two points in the middle of the cylinder $(\frac{1}{2},\phi_1)$ and $(\frac{1}{2},\phi_2)$ with $0< \phi_1 - \phi_2 < \frac{1}{2}$:
\end{minipage}
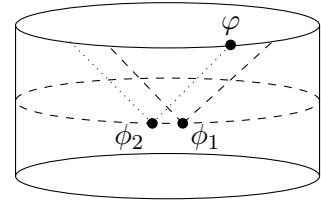
\begin{minipage}[t]{0.32\textwidth}
\vspace{0cm}
\hspace{0.5 cm}
\begin{tikzpicture}
% Draw the top ellipse
\draw (0,0) ellipse (2 and 0.3);
% Draw the sides
\draw (-2,0) -- (-2,-2);
\draw (2,0) -- (2,-2);
% Draw the bottom ellipse
\draw (0,-2) ellipse (2 and 0.3);
% Draw the middle horizontal line for reference
\draw[dashed] (0,-1) ellipse (2 and 0.3);
% Draw the points in the middle
\fill (0.2,-1.3) circle (2pt);
\fill (-0.2,-1.3) circle (2pt);
% Label the points
\node at (0.5,-1.48) {\(\phi_1\)};
\node at (-0.5,-1.48) {\(\phi_2\)};
% Draw the triangle
\draw[dashed] (0.2,-1.3) -- (-0.8,-0.25);
\draw[dashed] (0.2,-1.3) -- (1.4,-0.2);
\draw[dotted] (-0.2,-1.3) -- (0.83,-0.26);
\draw[dotted] (-0.2,-1.3) -- (-1.3,-0.2);
\fill (0.83,-0.26) circle (2pt);
\node at (0.83,0) {\(\varphi\)};
\end{tikzpicture}
\end{minipage}

\begin{align*}
    \dd_{\tau^p} \left((\frac{1}{2},\phi_1),(\frac{1}{2},\phi_2)\right) &= \sup_{z\in Cyl^p} \left|\tau^p\left((\frac{1}{2},\phi_1),z\right) - \tau^p \left((\frac{1}{2},\phi_2),z\right)\right|\\
    &\geq \left|\tau^p\left((\frac{1}{2},\phi_1),(1,\phi_2+\frac{1}{2})\right) - \tau^p \left((\frac{1}{2},\phi_2),(1,\phi_2+\frac{1}{2})\right)\right| \\
    &\geq \left|  \left( \frac{1}{2} - \left( \frac{1}{2} - (\phi_1 - \phi_2)) \right)^p\right)^{1/p} - 0\right| \quad \overset{p \to \infty}{\longrightarrow} \quad \frac{1}{2}
\end{align*}
Hence the Noldus distance for any two points in the middle of the cylinder grows to $\frac{1}{2}$ for $p\to \infty$. By thfe triangle inequality, these two points can not be inside the same ball of radius $\frac{1}{4}$ for $p$ large enough.
The number of $\frac{1}{4}$-balls to cover $Cyl^p$ is thereby unbounded for $p\to\infty$, meaning that $\{Cyl^p\}_p$ is not uniformly totally bounded with the Noldus metric.\\
For a given $p\in \N$, let $N_p$ be the number of $\frac{1}{8}$ - balls to cover $Cyl^p$. We then find a $q_p\in\N$ such that for all $q\geq q_p$, the space $Cyl^q$ can not be covered by $N_p$ balls of radius $\frac{1}{8}$. Hence
$\Phi^*\dd^+_{GH}(Cyl^p,Cyl^q) \geq \frac{1}{16}$.\\
Using $\frac{1}{2}\Phi^*\dd_{GH}^+ \leq \dd_{GH}^-$, for any $p\in \N$ and $q\geq q_p$, $\dd^-_{GH}(Cyl^p,Cyl^q)\geq \frac{1}{32}$ and therefore\\ $\{Cyl^p \pp p\in \N\}$ can not contain any converging subsequence.
\end{proof}

\subsection{Volume, Dimension and Ricci curvature}
Finally, we want to take a look at the timelike Ricci curvature of $Cyl^p$ in the sense of $TCD$ conditions. \\
As mentioned earlier, these conditions need a background measure. The easiest choice for this is to inherit the volume measure from the Euclidean background structure of $Cyl^p$. Another option is to calculate the Hausdorff measure and use this for the $TCD$ condition.
\begin{lemma}
\begin{equation*}
    \dim^\tau(Cyl^p) = \min\{p,2\} \text{  and  } \nu^{ \min\{p,2\}}(Cyl)=\begin{cases}
        0 & p<2\\
        2 \pi & p\geq 2
    \end{cases}
\end{equation*}
\end{lemma}
\begin{remark}
    This is a bit surprising, comparing it to the metric setting. Here the Hausdorff dimension can be larger than the (expected) vector space dimension but not smaller.
\end{remark}
\begin{proof}
For our argument, we will consider two scenarios. In the first, we split a given diamond into $n$ tilted diamonds (first graphic). In the second one we split it evenly into $k^2$ symmetric diamonds (second graphic). We can then combine both separation processes (third graphic).
    \begin{center}
        \begin{tikzpicture}
            \draw[rotate=45] (-1,-1) rectangle (1,1);
            \draw[rotate=45] (-1,0.5) -- (1,0.5);
            \draw[rotate=45] (-1,0.25) -- (1,0.25);
            \draw[rotate=45] (-1,0.75) -- (1,0.75);
            \draw[rotate=45] (-1,-0.25) -- (1,-0.25);
            \draw[rotate=45] (-1,-0.75) -- (1,-0.75);
            \draw[rotate=45] (-1,-0.5) -- (1,-0.5);
            \draw[rotate=45] (-1,0) -- (1,0);
            \coordinate (A) at (0,-1.4);
            \coordinate (B) at (-1.4,0); 
            \draw[thick, decorate, decoration={brace,mirror, amplitude=10pt}]
            (B) -- (A) node[midway, left, xshift=-7pt, yshift=-7pt,rotate=45] {$n$};
        \end{tikzpicture}   
        $\quad$
         \begin{tikzpicture}
            \draw[rotate=45] (-1,-1) rectangle (1,1);
            \draw[rotate=45] (-1,0) -- (1,0);
            \draw[rotate=45] (-1,0.5) -- (1,0.5);
            \draw[rotate=45] (-1,-0.5) -- (1,-0.5);
            \draw[rotate=45] (0.5,-1) -- (0.5,1);
            \draw[rotate=45] (-0.5,-1) -- (-0.5,1);
            \draw[rotate=45] (0,-1) -- (0,1);
            \coordinate (A) at (0,-1.4);
            \coordinate (B) at (-1.4,0); 
            \draw[thick, decorate, decoration={brace,mirror, amplitude=10pt}]
            (B) -- (A) node[midway, left, xshift=-7pt, yshift=-7pt,rotate=45] {$k$};
        \end{tikzpicture}   
        $\quad$
         \begin{tikzpicture}
            \draw[rotate=45] (-1,-1) rectangle (1,1);
            \draw[rotate=45] (-1,0) -- (1,0);
            \draw[rotate=45] (-1,0.5) -- (1,0.5);
            \draw[rotate=45] (-1,-0.5) -- (1,-0.5);
            \draw[rotate=45] (0.5,-1) -- (0.5,1);
            \draw[rotate=45] (-0.5,-1) -- (-0.5,1);
            \draw[rotate=45] (0,-1) -- (0,1);
\draw[rotate=45] (-1,0.75) -- (1,0.75);
\draw[rotate=45] (-1,0.875) -- (1,0.875);
\draw[rotate=45] (-1,0.625) -- (1,0.625);
\draw[rotate=45] (-1,0.5) -- (1,0.5);
\draw[rotate=45] (-1,0.375) -- (1,0.375);
\draw[rotate=45] (-1,0.25) -- (1,0.25);
\draw[rotate=45] (-1,0.125) -- (1,0.125);
\draw[rotate=45] (-1,0) -- (1,0);
\draw[rotate=45] (-1,-0.125) -- (1,-0.125);
\draw[rotate=45] (-1,-0.25) -- (1,-0.25);
\draw[rotate=45] (-1,-0.375) -- (1,-0.375);
\draw[rotate=45] (-1,-0.5) -- (1,-0.5);
\draw[rotate=45] (-1,-0.625) -- (1,-0.625);
\draw[rotate=45] (-1,-0.75) -- (1,-0.75);
\draw[rotate=45] (-1,-0.875) -- (1,-0.875);
            \draw[rotate=45] (-1,0) -- (1,0);
        \end{tikzpicture}   
    \end{center}
    We now define the following function, which measures the pre-Hausdorff volume with respect to the tiling of \( J((0,0),(1,0)) \) into \( n \cdot k^2 \) smaller diamonds, as described above.
\begin{equation*}
    v(p,d,n,k) \mDef  n k^2 \tau^p((0,0),(\frac{1}{2kn}+\frac{1}{2k},\frac{1}{2kn}-\frac{1}{2k}))^d
    =  n \left(\frac{1}{k}\right)^{d-2} \left(\left(\frac{1}{2 n}+\frac{1}{2}\right)^p-\left(\frac{1}{2}-\frac{1}{2 n}\right)^p\right)^{d/p}
\end{equation*}
Expanding $v$ in $\frac{1}{2 n}$, we observe that 
\begin{equation*}
    \lim_{n\to \infty} v(p,d,n,k) = \begin{cases}
        0 & \text{ for }p<d\\
        \left(\frac{1}{k}\right)^{d-2} 2^{1-p} p & \text{ for }p=d\\
        \infty & \text{ for }p>d
    \end{cases}
\end{equation*}
 We will consider the cases $p<2$ and $p>2$ individually. The case $p=2$ refers to Minkowski Space and is therefore trivial.\\
 \textbf{$\null\quad p<2$: } If we choose $d\leq p$, $v(p,d,n,k)$ is increasing in $n$ which means that we can cover $J((0,0),(1,0))$ most efficiently by symmetric diamonds ($n=1$).
 If we increase $k\to \infty$ (which we have to do in order to satisfy the condition $\text{diam}(J(p_i,q_i)) < \delta$ from the definition of the Hausdorff volume) we find that $v(p,d,1,k) \xrightarrow{k\to\infty} \infty$.\\
 For $d>p$, $v(p,d,n,k)$ is decreasing to $0$ for $n\to\infty$. This means we can cover a diamond most efficiently by tilted diamonds (large $n$) and get an area of $0$ for any finite $k$. This is preserved when taking the limit $k\to \infty$, meaning that $\lim_{k\to\infty}\lim_{n\to \infty}v(p,d,n,k) = 0$.\\
 Concluding this case, the Hausdorff dimension is $p$ but the Hausdorff measure is zero.\\
  \textbf{$\null\quad p>2$: } 
  For $d<2<p$, $v(p,d,n,k)$ is increasing in $n$, which means we consider the case $n=1$. If we increase $k$, we observe $v(p,d,1,k)\xrightarrow{k\to\infty} \infty$.\\
  For $d=2$, $v(p,d,n,k)$ is still increasing in $n$, but $v(p,2,1,k)\xrightarrow{k\to\infty} 2$. Hence the Hausdorff volume and dimension coincide with the case for $p=2$.\\
  For $2<d$, we see (as expected) that $v(p,d,n,k)\xrightarrow{k\to\infty} 0$ for all choices of $n$.
\end{proof}
\begin{remark}
        The above lemma shows that the functors $\mathcal{MCS}$ and $\mathcal{FLD}$ of \cite{M_ller_2023}, which form a bijection on spacetimes, cannot be extended as a bijection to the category of Lorentzian length spaces or Lorentzian metric spaces, since different Lorentzian spaces induce the same Hausdorff volume and causal relation.
\end{remark}
Knowing that both, the Hausdorff volume and dimension induced by $\tau$, as well as the volume and dimension induced by the background Minkowski space, are bounded from above, the family $\{Cyl^p\}_{p>2}$ serves as an interesting example in the context of a potential Ricci curvature–driven Gromov compactness result, as known in the Riemannian setting.\\
Using Proposition A.21 of \cite{beran2024nonlineardalembertcomparisontheorem}, we can conclude that for $p > 2$, the space $Cyl^p$ satisfies the condition $TCD_{\Tilde{p}}(0,2)$ for all $\Tilde{p} \in (0,1)$, with respect to both background measures mentioned above. Hence, the Ricci curvature can be considered to be bounded from below by $0$, and the dimension from above by $2$. The spaces have bounded volume and timelike diameter, yet they are Gromov–Hausdorff divergent.\\
This indicates that additional assumptions are still required for a Ricci curvature–driven Gromov compactness result. For instance, such a result might hold on Lorentzian manifolds but fail in the more general setting of Lorentzian Finsler manifolds.

\section{Acknowledgements}
I would like to thank my supervisor Olaf Müller for his feedback and valuable input
at various stages throughout this project. This research was supported by the Austrian Science Fund
(FWF) [Grants DOI 10.55776/PAT1996423 and 10.55776/EFP6]. For open access purposes,
I has applied a CC BY public copyright license to any author-accepted manuscript version arising from this submission.
\newpage
%a\bibliographystyle{unsrt}
\bibliographystyle{plain}
\bibliography{lit.bib}
\end{document}